\documentclass[11pt,a4paper]{article}
\usepackage[T1]{fontenc}
\usepackage[utf8]{inputenc}
\usepackage{lmodern}
\usepackage{amsmath,amssymb,amsthm,mathtools}
\usepackage[margin=28mm]{geometry}
\usepackage{microtype}
\usepackage{booktabs,array,tikz,needspace}
\usepackage{flafter}
\usepackage[hidelinks,unicode]{hyperref}
\hypersetup{pdftitle={Polya's conjecture for Dirichlet eigenvalues of spherical shells and Aharonov-Bohm annuli},
  pdfauthor={Yutian Li},
  pdfsubject={Translated-tail lattice estimates, strict eigenvalue inequalities, and arbitrary magnetic flux}}
\numberwithin{equation}{section}
\newtheorem{theorem}{Theorem}[section]
\newtheorem{lemma}[theorem]{Lemma}
\newtheorem{proposition}[theorem]{Proposition}
\newtheorem{corollary}[theorem]{Corollary}
\theoremstyle{remark}
\newtheorem{remark}[theorem]{Remark}
\allowdisplaybreaks[1]
\newcommand{\wmin}{w_{\min}}
\title{P\'olya's conjecture for Dirichlet eigenvalues of\\spherical shells and Aharonov--Bohm annuli}
\author{Yutian Li\\[3pt]
  \small School of Mathematics and Statistics\\
  \small Nanfang College, Guangzhou 510970, Guangdong, China\\
  \small\href{mailto:yutianlee@gmail.com}{\texttt{yutianlee@gmail.com}}}
\date{20 October 2026}
\begin{document}
\maketitle
\begin{abstract}
We establish strict Dirichlet P\'olya inequalities for spherical
shells in every dimension \(d\ge2\) and for every radius ratio.
The inclusive eigenvalue counting function lies strictly below
the leading Weyl term at every positive energy; consequently,
each eigenvalue strictly exceeds its Weyl prediction, with
multiplicities included. We also prove the same strict inequality
for the Aharonov--Bohm operator on planar annuli at arbitrary flux.
In dimension two, the nonmagnetic result gives a fully analytic
proof of the planar annulus theorem of Filonov, Levitin,
Polterovich and Sher (J. Lond. Math. Soc., 2026). The common ingredient is a lattice-point
bound for the shell action, uniform in the starting point of a tail
and in the lattice translation. Quadrature for the inverse action
controls the change of convexity caused by the inner boundary.
The tail estimate supplies
the hypothesis of the known dimension-reduction argument for
balls, while lattice translations represent the magnetic flux.
Laptev's product argument yields corresponding strict bounds on
products of shells and nonempty bounded open sets.
\end{abstract}



\section{Introduction and main results}\label{sec:1}

Let \(\Omega\subset\mathbb R^d\) be a bounded domain, and let
\(\lambda_1^D(\Omega)\le\lambda_2^D(\Omega)\le\cdots\) be the
eigenvalues of its Dirichlet Laplacian, listed with multiplicity. We
use the \emph{inclusive} counting function
\[
N_D(\Omega,\Lambda)=\#\{j\ge1:\lambda_j^D(\Omega)\le\Lambda\}.
\]
Write \(\omega_d=\pi^{d/2}/\Gamma(d/2+1)\) and
\(L_{0,d}=\omega_d/(2\pi)^d\). By Weyl's law,
\(N_D(\Omega,\Lambda)=L_{0,d}|\Omega|\Lambda^{d/2}+o(\Lambda^{d/2})\).
P\'olya's conjecture \cite{Polya54} asserts that the leading term
is an upper bound at every positive energy:
\begin{equation}\label{eq:polya}
N_D(\Omega,\Lambda)\le L_{0,d}|\Omega|\Lambda^{d/2}\qquad(\Lambda>0).
\end{equation}
The Neumann counterpart asks for the reverse bound for the
Neumann counting function.

P\'olya proved \eqref{eq:polya} for domains that tile \(\mathbb R^d\)
\cite{Polya61}, together with its Neumann counterpart under an
additional assumption that Kellner later removed \cite{Kellner}.
Laptev showed that \eqref{eq:polya} passes from
\(\Omega_1\subset\mathbb R^{d_1}\), \(d_1\ge2\), to every product
\(\Omega_1\times\Omega_2\) with \(\Omega_2\) bounded
\cite[Theorem~2.8]{Laptev}. For general domains the Berezin--Li--Yau
inequality \cite{Berezin,LiYau} gives \eqref{eq:polya} with the
right-hand side multiplied by \((1+2/d)^{d/2}\).

Among domains that do not tile, Filonov, Levitin, Polterovich and Sher proved
\eqref{eq:polya} for disks, circular sectors and Euclidean balls of
every dimension \cite{FLPSballs}, and for planar annuli
\cite{FLPS2026}; their proof for annuli combines analytic estimates
with a rigorous computer-assisted verification on a bounded set of
parameters. They also proved the analogous inequality for the
Aharonov--Bohm operator on the disk at every flux
\cite{FLPSmagnetic}, answering a question of Frank and Hansson
\cite{FrankHansson}.

We prove \eqref{eq:polya}, with strict inequality, for spherical shells
in every dimension and for Aharonov--Bohm annuli at every flux. Let
\(A^d_{r,R}=\{x\in\mathbb R^d:r<|x|<R\}\), where \(0<r<R\).

\begin{theorem}[Spherical shells]\label{thm:shell}
For every integer \(d\ge2\) and every \(\Lambda>0\),
\begin{equation}\label{eq:1.1}
N_D(A^d_{r,R},\Lambda)<
\frac{R^d-r^d}{2^d\Gamma(d/2+1)^2}\Lambda^{d/2}.
\end{equation}
Equivalently, for every integer \(j\ge1\),
\[
\lambda_j^D(A^d_{r,R})>
\left(\frac{2^d\Gamma(d/2+1)^2j}{R^d-r^d}\right)^{2/d}.
\]
\end{theorem}

The right-hand side of \eqref{eq:1.1} is
\(L_{0,d}|A^d_{r,R}|\Lambda^{d/2}\); both sides vanish at
\(\Lambda=0\). Because the count is inclusive, evaluation of
\eqref{eq:1.1} at \(\Lambda=\lambda_j^D\) gives the eigenvalue bound,
multiplicities included. For \(d=2\), Theorem~\ref{thm:shell} is the
theorem of \cite{FLPS2026}; we give an analytic proof.
For \(d\ge3\) it extends the Dirichlet theorem for balls
\cite{FLPSballs} to shells of every radius ratio. The two-term
asymptotics of Guo, Jiang, Wang and Yang \cite{GJWY} imply
\eqref{eq:1.1} only above an unspecified threshold.

Our second result concerns magnetic flux through the hole of a planar
annulus. For \(\alpha\in\mathbb R\), define
\[
\mathbf A_0(x)=\frac{(-x_2,x_1)}{|x|^2},\qquad
H_\alpha=(-i\nabla-\alpha\mathbf A_0)^2
\quad\text{on }A^2_{r,R},
\]
with Dirichlet boundary conditions, defined by its closed quadratic
form on \(H_0^1(A^2_{r,R})\). The magnetic potential
\(\alpha\mathbf A_0\) is smooth on the closed annulus and its
magnetic field vanishes there. Its circulation
around the hole is \(2\pi\alpha\), so \(\alpha\) is the normalized flux.
Denote the eigenvalues, with multiplicities,
by \(\lambda_j^\alpha\) and their inclusive count by \(N_\alpha\).

\begin{theorem}[Aharonov--Bohm annuli]\label{thm:magnetic}
For every \(\alpha\in\mathbb R\) and \(\Lambda>0\),
\[
N_\alpha(A^2_{r,R},\Lambda)<\frac{R^2-r^2}{4}\Lambda,
\qquad
\lambda_j^\alpha(A^2_{r,R})>\frac{4j}{R^2-r^2}\quad(j\ge1).
\]
\end{theorem}

Theorem~\ref{thm:magnetic} answers the question raised in
\cite[Remark~1.5]{FLPSmagnetic} for the Dirichlet problem on annuli.
Frank and Hansson had checked the Dirichlet inequality numerically
for some annuli \cite{FrankHansson}.
The disk proof of \cite{FLPSmagnetic} uses a translated-lattice
estimate for nonnegative decreasing convex \(1/2\)-Lipschitz functions
\cite[Theorem~2.3]{FLPSmagnetic}. The shell action is not convex, and
Theorem~\ref{thm:tail} below replaces that estimate. For integer
\(\alpha\) the operator \(H_\alpha\) is unitarily equivalent to the
Dirichlet Laplacian, and Theorem~\ref{thm:magnetic} reduces to the
planar case of Theorem~\ref{thm:shell}.

Laptev's product argument gives strict Dirichlet bounds on
\(A^d_{r,R}\times\Omega_2\) for every \(d\ge2\) and every nonempty
bounded open set \(\Omega_2\subset\mathbb R^k\). It also gives the
corresponding magnetic bounds on \(A^2_{r,R}\times\Omega_2\), with
the flux acting in the annular factor (Corollary~\ref{cor:products}).
For example, the strict inequality for every cylindrical pipe
\(\{x\in\mathbb R^3:r<(x_1^2+x_2^2)^{1/2}<R,\ 0<x_3<h\}\)
with \(h>0\) already follows from the planar annulus theorem
\cite{FLPS2026} and Laptev's result \cite{Laptev}.

\subsection{The lattice estimate}

The proof follows the scheme of \cite{FLPSballs,FLPS2026}.
After scaling to \(R=1\), put \(K=\sqrt\Lambda\) and \(\mu=rK\).
Let \(G_K\) and \(G_\mu\) denote the ball actions defined in
Section~\ref{sec:2}; the shell action is \(A=G_K-G_\mu\).
For angular order \(\nu\ge0\), let \(q(\nu)\) count the integers
\(n\ge1\) satisfying \(n-1/4<A(\nu)\).
The Bessel phase bound in Lemma~\ref{lem:phase} shows that at most
\(q(\nu)\) radial eigenvalues lie at or below \(K^2\).
Thus the spectral counting function is bounded by a weighted
count of lattice points under the graph of \(A\).
For the ball, \(G_K\) is decreasing, convex
and \(1/2\)-Lipschitz, and the lattice inequality
\cite[Theorem~5.1]{FLPSballs} applies to it and to all of its tails.
For the shell, \(A\) is concave on \((0,\mu)\) and convex on
\((\mu,K)\). In \cite{FLPS2026} the corresponding count is estimated
by trapezoidal floor sums on one, two or three subintervals of
\([0,\lfloor K\rfloor+1]\), depending on the region of parameters, and
a rigorous computation treats a remaining bounded parameter region.

The main new ingredient is Theorem~\ref{thm:tail}, a lattice-point
inequality for the shell action that is uniform in the starting point
\(a\ge0\) of a tail and in a translation \(\alpha\in\mathbb R\) of the
lattice. We prove it using the inverse \(x=A^{-1}\) of \(A\) on \([0,K]\).
The count becomes a quadrature of \(x\) at the points \(j-1/4\), and
the quadrature error splits into contributions of unit cells
(Lemma~\ref{lem:3.2}). The inverse is convex where \(x>\mu\) and
concave where \(x<\mu\). A cell contained in one branch has an
explicit positive margin; only the first cell and the unique cell that
crosses the change of branch need separate estimates. Uniformity in
\(a\) lets the dimension-reduction argument of \cite[\S7]{FLPSballs}
recover the exact Weyl constant from the spherical-harmonic
multiplicities (Lemma~\ref{lem:angular}); uniformity in \(\alpha\)
gives the magnetic theorem.

\subsection{Related results and organisation}

Quantitative P\'olya-type inequalities for annular sectors were
obtained by Dai and Sun \cite{DSsectors}; those sectors retain their
radial boundary faces, including at aperture \(2\pi\), whereas the
angular conditions here are periodic or magnetic.

Recent preprints establish the Neumann counterpart for higher-dimensional
balls, independently by Li \cite{LiNeumann} and by Filonov, Levitin,
Polterovich and Sher \cite{FLPSNeumann}, and for planar annuli by
Dai and Sun \cite{DSNeumann}.
A companion preprint \cite{LiNeumannShells} proves the strict
Neumann P\'olya inequality for concentric spherical shells in
dimensions \(d\ge3\).
Together with the planar annulus theorem of Dai and Sun, that result
and the present theorem establish strict P\'olya inequalities for
both boundary conditions on spherical shells in every dimension
\(d\ge2\).

Tang and Wang \cite{TW} have recently announced a proof of the strict
Dirichlet inequality for bounded simply connected planar domains.

Section~\ref{sec:2} states the translated-tail theorem and proves the
spectral reduction for real orders. Assuming the tail theorem,
Section~\ref{sec:4} proves both main theorems, including their
low-frequency endpoints. The analytic proof of the tail estimate then
occupies Sections~\ref{sec:3}--\ref{sec:6}: inverse-action quadrature
in Section~\ref{sec:3}, geometric estimates in Section~\ref{sec:5},
and positivity of the quadrature errors in Section~\ref{sec:6}.
Section~\ref{sec:applications} gives product bounds, an enhanced
Dirichlet bound for each fixed shell, and examples with small cubical
holes, followed by quantitative questions. Appendix~\ref{app:shape}
contains the moment proof of the inverse-slope shape used in the
quadrature argument.

\section{The translated-tail estimate and spectral reduction}\label{sec:2}

\subsection{The shell action}

Scale to outer radius one and write
\(\rho=r/R\), \(K=R\sqrt\Lambda\), and \(\mu=\rho K\).
The count vanishes when \(K=0\). For \(z\ge0\), the radial action is
\[
A(z)=\frac1\pi\int_\mu^K
\sqrt{\left(1-\frac{z^2}{s^2}\right)_+}\,ds,
\qquad u_+:=\max(u,0).
\]
This is the available radial phase space at angular order \(z\).
Integration gives \(A=G_K-G_\mu\), where, for \(L>0\),
\[
G_L(z)=\frac{\sqrt{L^2-z^2}-z\arccos(z/L)}{\pi}\quad(0\le z\le L),
\qquad G_L(z)=0\quad(z>L),
\]
\begin{equation}\label{eq:2.1}
[s]_<:=\max(0,\lceil s\rceil-1),
\qquad q(z)=[A(z)+1/4]_<,\qquad T=A(0)=\frac{K-\mu}{\pi}.
\end{equation}

The function \(A\) decreases from \(T\) to zero on \([0,K]\). Its graph changes curvature at \(z=\mu\): for \(\mu<z<K\) the radial turning point lies inside the shell, whereas for \(z<\mu\) the whole radial interval contributes.

\begin{theorem}[Uniform translated-tail estimate]\label{thm:tail}
Suppose \(0<\mu<K\) and \(T=(K-\mu)/\pi>1\). Then
\begin{equation}\label{eq:tail}
\sum_{m\in\mathbb Z}q(a+|m+\alpha|)
\le 2\int_a^\infty A(z)\,dz
\qquad(a\ge0,\ \alpha\in\mathbb R),
\end{equation}
with strict inequality whenever \(a<K\).
\end{theorem}

For \(\alpha=0\), the summands have the planar angular weights
\(1,2,2,\ldots\). Varying \(a\) supplies the tails needed for
higher-dimensional angular multiplicities. Taking \(a=0\) and varying
\(\alpha\) supplies the magnetic count. The quadrature reduction is
given in Section~\ref{sec:3}, and the required estimates are proved
in Sections~\ref{sec:5}--\ref{sec:6}.

\subsection{Real-order spectral reduction}

For real \(\nu\ge0\) and \(s>0\) write
\(J_\nu(s)+iY_\nu(s)=M_\nu(s)e^{i\vartheta_\nu(s)}\), where
\(M_\nu>0\) and the phase \(\vartheta_\nu\) is continuous with
\(\vartheta_\nu(0+)=-\pi/2\).
The following estimate for the phase increment across the shell
follows the argument of
\cite[Lemma~5.2, (5.6) and (5.8)]{FLPS2026}.
We include the proof for every real \(\nu\ge0\), since the
higher-dimensional and magnetic applications require half-integer
and arbitrary real orders.
Sharper two-sided enclosures of the phase are given in
\cite{FLPSphase}.

\begin{lemma}[Phase bound]\label{lem:phase}
Let \(0<\rho<1\), \(K>0\), \(\mu=\rho K\), and \(\nu\ge0\). Then
\begin{equation}\label{eq:2.4}
G_K(\nu)-G_\mu(\nu)
<
\gamma_{\rho,K}(\nu):=
\frac{\vartheta_\nu(K)-\vartheta_\nu(\mu)}{\pi}
<G_K(\nu)-G_\mu(\nu)+\frac14 .
\end{equation}
\end{lemma}

\begin{proof}
For \(x>0\), put
\(\delta_\nu(x)=G_x(\nu)-\vartheta_\nu(x)/\pi\), using the convention
\(G_x(\nu)=0\) for \(x\le\nu\). The Wronskian gives
\[
\vartheta_\nu'(x)=
\frac{2}{\pi x\{J_\nu(x)^2+Y_\nu(x)^2\}}>0.
\]
Hence \(\delta_\nu'<0\) for \(0<x<\nu\). For \(x>\nu\),
\(\partial_xG_x(\nu)=\sqrt{x^2-\nu^2}/(\pi x)\), while
\(\vartheta_\nu'(x)>\sqrt{x^2-\nu^2}/x\) by
\cite[Eq.~(56)]{Horsley}. Thus \(\delta_\nu'<0\) there as well,
and continuity at \(x=\nu\) shows that \(\delta_\nu\) is strictly
decreasing on \((0,\infty)\).
The phase normalization at zero and the large-\(x\) expansion
\cite[(10.18.18)]{DLMF} give, for each fixed \(\nu\),
\[
\lim_{x\downarrow0}\delta_\nu(x)=\frac12,
\qquad
\lim_{x\to\infty}\delta_\nu(x)=\frac14.
\]
Therefore
\[
0<
\gamma_{\rho,K}(\nu)-G_K(\nu)+G_\mu(\nu)
=\delta_\nu(\mu)-\delta_\nu(K)
<\frac14,
\]
as required.
\end{proof}

The angular orders are \(\nu_{\ell,d}=\ell+(d-2)/2\).
For \(d\ge3\), their multiplicities are
\[
\kappa_{d,\ell}=\frac{(2\ell+d-2)(\ell+d-3)!}{\ell!(d-2)!};
\]
in dimension two, the multiplicities are \(1,2,2,\ldots\).

\begin{proposition}[Spectral reduction]\label{prop:2.1}
For every real \(\nu\ge0\), the number of radial Dirichlet eigenvalues
at or below \(K^2\) is at most \(q(\nu)\). Consequently,
\begin{equation}\label{eq:2.2}
N_D(A^d_{\rho,1},K^2)\le P_d:=
\sum_{\ell\ge0}\kappa_{d,\ell}q(\nu_{\ell,d}).
\end{equation}
\end{proposition}

\begin{proof}
Separation of variables gives the eigenfunctions
\(s^{1-d/2}\bigl(c_1J_\nu(ks)+c_2Y_\nu(ks)\bigr)\,\mathcal Y_\ell\),
with \(\mathcal Y_\ell\) a spherical harmonic of degree \(\ell\),
\(\nu=\nu_{\ell,d}\) and eigenvalue \(k^2\). For a radial factor \(f\),
the substitution \(f=s^{-(d-1)/2}v\) turns the radial equation into
\(-v''+(\nu^2-1/4)s^{-2}v=k^2v\), so the radial Dirichlet operator is
\[
L_\nu=-\frac{d^2}{ds^2}+\frac{\nu^2-1/4}{s^2}
\quad\text{on }(\rho,1).
\]
Write \(\lambda_{\nu,n}\), \(n\ge1\), for its Dirichlet
eigenvalues in increasing order.
Writing \(v=\sqrt{s}\,u\), its quadratic form becomes
\[
\int_\rho^1\left(|v'|^2+
       \frac{\nu^2-1/4}{s^2}|v|^2\right)\,ds
=
\int_\rho^1\left(s|u'|^2+
       \frac{\nu^2}{s}|u|^2\right)\,ds>0
\]
for \(0\ne u\in H_0^1(\rho,1)\); thus all radial eigenvalues
are positive.
For \(\nu\ge1/2\), min--max comparison with an interval gives
\begin{equation}\label{eq:2.3}
\lambda_{\nu,n}\ge\frac{n^2\pi^2}{(1-\rho)^2}+\nu^2-\frac14.
\end{equation}

The radial eigenfrequencies of order \(\nu\) are the positive zeros
\(k\) of the cross product
\(J_\nu(k)Y_\nu(\rho k)-Y_\nu(k)J_\nu(\rho k)
=-M_\nu(k)M_\nu(\rho k)\sin(\vartheta_\nu(k)-\vartheta_\nu(\rho k))\).
Nicholson's integral \cite[(10.9.30)]{DLMF},
\[
M_\nu(s)^2=\frac8{\pi^2}\int_0^\infty
K_0(2s\sinh t)\cosh(2\nu t)\,dt,
\]
shows that \(M_\nu(s)^2\) strictly decreases for \(s>0,\nu\ge0\), since the modified Bessel function \(K_0\) is positive and strictly decreasing. The Wronskian therefore gives
\[
\frac{d}{dk}\{\vartheta_\nu(k)-\vartheta_\nu(\rho k)\}
=\frac2{\pi k}\left\{M_\nu(k)^{-2}-M_\nu(\rho k)^{-2}\right\}>0.
\]
The phase difference increases from zero to infinity, so its levels \(\pi,2\pi,\ldots\) index the radial eigenfrequencies. Thus the \emph{inclusive} radial count is
\(\lfloor\gamma_{\rho,K}(\nu)\rfloor\). The strict phase bound
\eqref{eq:2.4} gives
\[
\lfloor\gamma_{\rho,K}(\nu)\rfloor
\le \lceil A(\nu)+1/4\rceil-1=q(\nu),\qquad \nu\ge0.
\]
For \(\nu\ge K\), \(A(\nu)=q(\nu)=0\), so no radial eigenvalue
contributes and all angular sums are finite. The spectral
count is inclusive, while the action count uses the strict bracket;
the strict phase inequality is exactly what makes these conventions
compatible. \end{proof}

\section{Proofs of the main theorems}\label{sec:4}

We first show how tail bounds produce the exact Weyl coefficient
in every dimension. We then apply the independent lattice
translation to magnetic annuli.

\subsection{Angular lifting}

The following lemma is the dimension-reduction argument of
\cite[\S7]{FLPSballs} (Theorem~7.1, inequality~(7.6) and Lemma~7.2
there), written so that it uses only tail bounds. We include the proof
for completeness; the new input for shells is Theorem~\ref{thm:tail}.

\begin{lemma}[Angular lifting]\label{lem:angular}
Let \(A\) be nonnegative, decreasing and absolutely continuous on
\([0,K]\), with \(A(K)=0\), and extend it by zero afterwards.
Suppose \(q\ge0\) vanishes on \([K,\infty)\) and satisfies
\[
q(a)+2\sum_{j\ge1}q(a+j)\le2\int_a^K A(z)\,dz
\qquad(a\ge0).
\]
Then for every integer \(d\ge2\),
\[
\sum_{\ell\ge0}\kappa_{d,\ell}q(\ell+(d-2)/2)
\le\frac2{(d-2)!}\int_0^Kz^{d-2}A(z)\,dz.
\]
The inequality is strict if the assumed tail bound is strict at
\(a=(d-2)/2<K\).
\end{lemma}

\begin{proof}
Let \(p=d-2\ge0\) and \(b=p/2\), and denote the sum by \(P_d\).
Define the nonnegative coefficients \(c_m\) by
\[
(1-u)^{-p}=\sum_{m\ge0}c_m u^m.
\]
For \(p=0\), only \(c_0=1\) is nonzero. In every dimension the angular multiplicities satisfy
\[
\kappa_{d,\ell}=c_\ell+2\sum_{m<\ell}c_m.
\]
The assumed tail bound therefore gives
\begin{equation}\label{eq:4.1}
P_d\le2\sum_{m\ge0}c_m\int_{b+m}^K A(z)\,dz
=2\int_0^K A(z)S(z)\,dz,\qquad
S(z)=\sum_{b+m\le z}c_m.
\end{equation}
Put \(F(z)=\int_0^z S=\sum_m c_m(z-b-m)_+\). It vanishes for \(z<b\).
For \(z=b+L+\tau\), where \(L\in\mathbb N_0\) and \(0\le\tau<1\), binomial summation gives
\begin{equation}\label{eq:4.2}
(p+1)!F(z)=\bigl(L+(p+1)\tau\bigr)\prod_{i=1}^p(L+i)
\le z^{p+1}.
\end{equation}
The inequality follows from the arithmetic--geometric mean
inequality, since the \(p+1\) factors have arithmetic mean \(z\).
The empty product makes the formula valid for \(p=0\).
With \(\sigma=-A'\), integrate by parts, apply
\eqref{eq:4.2}, and integrate by parts again to obtain
\begin{equation}\label{eq:4.3}
\int_0^K AS=\int_0^K\sigma F
\le\frac1{p!}\int_0^Kz^pA(z)\,dz.
\end{equation}
This proves the lemma; strictness follows from the \(m=0\) term,
whose coefficient is \(c_0=1\).
\end{proof}

\subsection{Spherical shells}\label{sec:shell-proof}

\begin{proof}[Proof of Theorem~\ref{thm:shell}]
First suppose \(T>1\). Theorem~\ref{thm:tail}, with \(\alpha=0\),
and Lemma~\ref{lem:angular} apply to the shell action.
Fubini's theorem, followed by the beta integral and gamma duplication
formula, gives
\begin{equation}\label{eq:4.4}
\frac2{(d-2)!}\int_0^Kz^{d-2}A(z)\,dz
=\frac{(K^d-\mu^d)B((d-1)/2,3/2)}{\pi d(d-2)!}
=\frac{K^d-\mu^d}{2^d\Gamma(d/2+1)^2}.
\end{equation}
Put \(b=(d-2)/2\). If \(K>b\), the \(m=0\) tail in \eqref{eq:4.1} is strictly
bounded, so \eqref{eq:2.2} gives the strict inequality for \(T>1\).
If \(K\le b\), the count is zero and the Weyl term is positive.

For \(T<1\), \eqref{eq:2.3} excludes every order
\(\nu\ge1/2\). At \(T=1\), the only possible equality in that
comparison is the first radial eigenvalue of order \(\nu=1/2\).
For \(d\ge4\), no order contributes when \(T\le1\). For \(d=3\),
only \(T=1\) needs attention: the order \(1/2\) has multiplicity one,
whereas
\[
\frac{2(K^3-\mu^3)}{9\pi}>\frac{2\pi^2}{9}>1.
\]
In dimension two, only the zero order may contribute, and
Proposition~\ref{prop:2.1} gives
\[
N_D(A^2_{\rho,1},K^2)\le q(0)=[T+1/4]_<\le1.
\]
The count is zero for \(T\le3/4\). For \(3/4<T\le1\),
\[
\frac{K^2-\mu^2}{4}\ge\frac{\pi^2T^2}{4}>
\frac{9\pi^2}{64}>1.
\]
These arguments include all spectral thresholds. Scaling proves
\eqref{eq:1.1}; evaluation at an eigenvalue proves the strict
individual inequality, including multiplicities.
\end{proof}

\subsection{Magnetic annuli}\label{sec:magnetic}

\begin{proof}[Proof of Theorem~\ref{thm:magnetic}]
After scaling to outer radius one, separation in the modes
\(e^{im\varphi}\) gives the real orders \(|m-\alpha|\),
\(m\in\mathbb Z\). Proposition~\ref{prop:2.1} and
Theorem~\ref{thm:tail}, with \(a=0\) and translation \(-\alpha\), give for \(T>1\)
\[
N_\alpha(A^2_{\rho,1},K^2)
\le\sum_{m\in\mathbb Z}q(|m-\alpha|)
<2\int_0^K A(z)\,dz=\frac{K^2-\mu^2}{4}.
\]

For \(T\le1\), \eqref{eq:2.3} excludes every order \(\nu>1/2\).
Reduce the flux to
\(\delta=\operatorname{dist}(\alpha,\mathbb Z)\in[0,1/2]\).
If \(\delta<1/2\), the only possible contributing order is
\(\delta\), and its count is at most
\(q(\delta)\le q(0)\le1\). A contribution requires
\(T>3/4\), so the planar low-frequency estimate in
Section~\ref{sec:shell-proof} applies. If \(\delta=1/2\), there is no
contribution for \(T<1\). At \(T=1\), the operator \(L_{1/2}=-d^2/ds^2\)
has its first eigenvalue at \(K^2=\pi^2/(1-\rho)^2\).
The two angular modes of order \(1/2\) therefore contribute
an inclusive count of \(2\), whereas
\[
\frac{K^2-\mu^2}{4}>\frac{\pi^2}{4}>2.
\]
Scaling and evaluation at spectral thresholds complete the proof.
\end{proof}

\begin{remark}
The conclusion also holds for every curl-free vector potential smooth
on the closed annulus. Subtracting \(\alpha\mathbf A_0\), with the same
circulation, leaves the gradient of a single-valued smooth function;
multiplication by its phase is a unitary gauge transformation.
In particular, an Aharonov--Bohm pole may lie anywhere strictly
inside the removed inner disk.
\end{remark}

\section{Inverse action and cell quadrature}\label{sec:3}

For \(a\ge0\) and \(\alpha\in\mathbb R\), define
\begin{equation}\label{eq:3.1}
D_A(a,\alpha)=2\int_a^K A(z)\,dz
-\sum_{m\in\mathbb Z}q(a+|m+\alpha|),
\end{equation}
where the integral is zero for \(a\ge K\). Theorem~\ref{thm:tail} asserts that this deficit is nonnegative,
and positive when \(a<K\). We reduce this assertion to a
quadrature estimate that does not depend on \(\alpha\).

The restriction of \(A\) to \([0,K]\) is strictly decreasing,
with range \([0,T]\). Let \(x=A^{-1}\) and \(c=A(\mu)\).
Then \(x(0)=K\), \(x(T)=0\), and \(x(c)=\mu\).
Extend \(x\) by zero for \(t>T\).
We call \(0<t<c\) the \textbf{outer branch}, where \(x(t)>\mu\),
and \(c<t<T\) the \textbf{inner branch}, where \(x(t)<\mu\).
The branches meet at \(t=c\), which we call their interface.
Each unit cell \([j-1,j]\), \(j\ge1\), is sampled at
\(j-1/4\); see Figure~\ref{fig:inverse}.

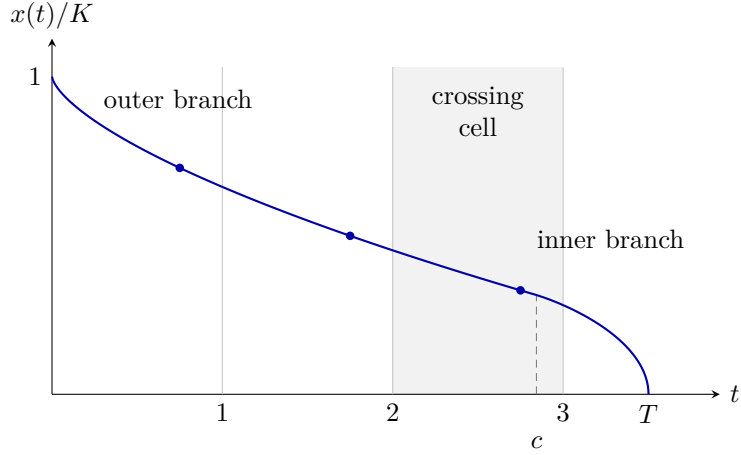
\begin{figure}[ht]
\centering
\begin{tikzpicture}[x=2.25cm,y=4.2cm,>=stealth,font=\small]
\fill[gray!10] (2,0) rectangle (3,1.03);
\draw[->] (0,0)--(3.92,0) node[right] {\(t\)};
\draw[->] (0,0)--(0,1.12) node[above] {\(x(t)/K\)};
\foreach \v in {1,2,3} {
  \draw[gray!45] (\v,0)--(\v,1.03);
  \node[below] at (\v,0) {\(\v\)};
}
\draw[densely dashed,gray] (2.84372,0)--(2.84372,0.31250);
\node[below,yshift=-12pt] at (2.84372,0) {\(c\)};
\node[below] at (3.50141,0) {\(T\)};
\node[left] at (0,1) {\(1\)};
\draw[thick,blue!65!black] plot coordinates {(0.00000,1.00000) (0.00589,0.98854) (0.01668,0.97708) (0.03066,0.96562) (0.04722,0.95417) (0.06604,0.94271) (0.08686,0.93125) (0.10952,0.91979) (0.13388,0.90833) (0.15985,0.89687) (0.18733,0.88542) (0.21625,0.87396) (0.24654,0.86250) (0.27816,0.85104) (0.31105,0.83958) (0.34517,0.82812) (0.38049,0.81667) (0.41696,0.80521) (0.45456,0.79375) (0.49327,0.78229) (0.53304,0.77083) (0.57387,0.75938) (0.61572,0.74792) (0.65858,0.73646) (0.70243,0.72500) (0.74725,0.71354) (0.79302,0.70208) (0.83974,0.69063) (0.88737,0.67917) (0.93592,0.66771) (0.98537,0.65625) (1.03570,0.64479) (1.08691,0.63333) (1.13899,0.62187) (1.19191,0.61042) (1.24568,0.59896) (1.30029,0.58750) (1.35572,0.57604) (1.41197,0.56458) (1.46903,0.55312) (1.52690,0.54167) (1.58556,0.53021) (1.64500,0.51875) (1.70523,0.50729) (1.76624,0.49583) (1.82801,0.48438) (1.89055,0.47292) (1.95385,0.46146) (2.01790,0.45000) (2.08271,0.43854) (2.14825,0.42708) (2.21454,0.41563) (2.28156,0.40417) (2.34931,0.39271) (2.41779,0.38125) (2.48699,0.36979) (2.55691,0.35833) (2.62755,0.34688) (2.69890,0.33542) (2.77096,0.32396) (2.84372,0.31250) (2.88319,0.30556) (2.91878,0.29861) (2.95191,0.29167) (2.98310,0.28472) (3.01263,0.27778) (3.04069,0.27083) (3.06743,0.26389) (3.09296,0.25694) (3.11736,0.25000) (3.14070,0.24306) (3.16304,0.23611) (3.18444,0.22917) (3.20493,0.22222) (3.22456,0.21528) (3.24335,0.20833) (3.26134,0.20139) (3.27855,0.19444) (3.29501,0.18750) (3.31074,0.18056) (3.32575,0.17361) (3.34006,0.16667) (3.35370,0.15972) (3.36667,0.15278) (3.37899,0.14583) (3.39066,0.13889) (3.40171,0.13194) (3.41214,0.12500) (3.42196,0.11806) (3.43118,0.11111) (3.43980,0.10417) (3.44783,0.09722) (3.45529,0.09028) (3.46217,0.08333) (3.46848,0.07639) (3.47423,0.06944) (3.47942,0.06250) (3.48405,0.05556) (3.48813,0.04861) (3.49166,0.04167) (3.49464,0.03472) (3.49708,0.02778) (3.49898,0.02083) (3.50033,0.01389) (3.50114,0.00694) (3.50141,0.00000)};
\fill[blue!65!black] (0.75000,0.71285) circle (1.6pt);
\fill[blue!65!black] (1.75000,0.49887) circle (1.6pt);
\fill[blue!65!black] (2.75000,0.32728) circle (1.6pt);
\node at (.74,.93) {outer branch};
\node at (3.28,.49) {inner branch};
\node[align=center] at (2.5,.89) {crossing\\cell};
\end{tikzpicture}
\caption{The inverse action for \(K=16\), \(\mu=5\). Dots mark the
included nodes \(j-1/4\). The outer branch is convex, the inner branch
is concave, and only one unit cell contains \(c\) in its interior. The plot
illustrates the geometry; all cell estimates are proved analytically.}
\label{fig:inverse}
\end{figure}

\begin{lemma}[Counting and quadrature]\label{lem:3.2}
Let \(a\ge0\) and \(\alpha\in\mathbb R\). If \(n=q(a)\ge1\), then
\begin{equation}\label{eq:3.2}
D_A(a,\alpha)>Q_n,\qquad
Q_n:=2\int_0^n x(t)\,dt
-2\sum_{j=1}^n x(j-1/4)-n.
\end{equation}
The same statement holds with \(G_K\) in place of \(A\),
using the corresponding inverse and counting functions.
\end{lemma}

\Needspace{6\baselineskip}
\begin{proof}
Write \(t_j=j-1/4\).
Since \(q\) is decreasing, only the levels \(1\le j\le n=q(a)\)
can contribute. At each such level, \(t_j<A(a)\), hence
\(x(t_j)>a\).
An open interval of length \(2r\) contains fewer than \(2r+1\)
integers, including when it contains none.
Applying this with \(r=x(t_j)-a>0\) gives
\[
\begin{aligned}
\sum_{m\in\mathbb Z}q(a+|m+\alpha|)
&=\sum_{j=1}^n\#\{m\in\mathbb Z:|m+\alpha|<x(t_j)-a\}\\
&<2\sum_{j=1}^n(x(t_j)-a)+n.
\end{aligned}
\]
The number of levels \(n=q(a)\) does not depend on \(\alpha\).
The layer-cake identity gives
\begin{equation}\label{eq:3.3}
\int_a^K A(z)\,dz
=\int_0^\infty(x(t)-a)_+\,dt
\ge\int_0^n(x(t)-a)\,dt.
\end{equation}
Combining these formulas proves the claim. The zero extension makes
the argument unchanged when \(n>T\). Only monotonicity and inversion
of the action are used, so the proof also applies to \(G_K\).
\end{proof}

Thus the lattice count is controlled by the error in sampling \(x\) at the three-quarter point of each unit cell. The following two estimates complete the proof.

\begin{proposition}[Positivity of the quadrature error]\label{prop:3.3}
Suppose \(T>1\) and \(\mu\ge1\). Then \(Q_n>0\) for every integer \(n\) with \(1\le n<T+1/4\).
\end{proposition}

\begin{proposition}[Small inner radius]\label{prop:3.4}
Suppose \(T>1\) and \(0<\mu\le1\). If \(a\ge0\) and
\(q(a)\ge1\), then \(D_A(a,\alpha)>0\) for every
\(\alpha\in\mathbb R\).
\end{proposition}

Propositions~\ref{prop:3.3} and~\ref{prop:3.4} are proved in
Section~\ref{sec:6} using the estimates of Section~\ref{sec:5}.
Intervals contained in one branch contribute positively.
If a later interval contains \(c\) in its interior, its possible
loss is controlled by the positive contribution from the first
interval. For \(\mu\le1\), comparison with the ball suffices.

\begin{proof}[Proof of Theorem~\ref{thm:tail}]
If \(q(a)=0\), monotonicity makes every translated count vanish,
and the integral in \eqref{eq:3.1} is nonnegative, strictly positive
for \(a<K\). Otherwise \(n=q(a)<T+1/4\), and
Lemma~\ref{lem:3.2} together with Propositions~\ref{prop:3.3}
and~\ref{prop:3.4} gives the result.
\end{proof}

\section{Geometry and quadrature estimates}\label{sec:5}

\subsection{Parameters and inverse geometry}\label{sec:5.1}

We use the following notation:
\[
\begin{array}{c|l}
K,\mu&\text{scaled outer and inner radii}\\
T=(K-\mu)/\pi&\text{height of the action curve}\\
c=A(\mu),\quad y=T-c&\text{interface height and inner-branch length}\\
x=A^{-1},\quad \sigma=-A',\quad w=-x'&\text{inverse action and its slope}\\
\theta=\arccos(\mu/K),\quad\wmin=\pi/\theta&\text{interface angle and minimal inverse slope}\\
\eta=\wmin/2-1&\text{positive margin per cell}.
\end{array}
\]
Thus \(\wmin>2\), and \(w=1/\sigma(x)\) for \(0<t<T\). Write
\begin{equation}\label{eq:5.1}
f_0(s)=\sin s-s\cos s,\qquad
c=\frac K\pi f_0(\theta)=\frac\mu\pi(\tan\theta-\theta).
\end{equation}
For \(0\le z<K\), the action slope is

\[
\sigma(z)=\frac1\pi
\left(\arccos(z/K)-\mathbf1_{\{z<\mu\}}\arccos(z/\mu)\right).
\]
It reaches its maximum \(\theta/\pi\) at \(z=\mu\), so \(w\) reaches its minimum \(\wmin\) at \(t=c\). The outer branch \(0<t<c\) has decreasing \(w\), and \(w\ge\wmin\) on both branches below \(T\).
On the outer branch, \(t=Kf_0(s)/\pi\), \(w(t)=\pi/s\), where \(s=\arccos(x(t)/K)\). Since
\[
\frac{f_0(s)}{s^3}=\int_0^1v^2\frac{\sin(sv)}{sv}\,dv
\]
decreases, so does \(t^{1/3}w(t)\). This observation also applies to the ball action \(G_K\), whose minimum inverse slope is \(2\).
On the inner branch,
\[
\sigma'(z)=\frac1\pi\left(
\frac1{\sqrt{\mu^2-z^2}}-\frac1{\sqrt{K^2-z^2}}\right)>0.
\]

\begin{lemma}[Shape of the inner inverse slope]\label{lem:5.1}
For \(c\le t<T\), define
\[
H(t):=\sqrt{T-t}\,w(t).
\]
Then \(H\) extends continuously to a positive, increasing, concave
function on \([c,T]\). It is strictly increasing and strictly concave
on \((c,T)\), with
\begin{equation}\label{eq:5.2}
H(c)=\wmin\sqrt y,\qquad
H(T)=P:=\sqrt{\frac{\mu K}{2T}}.
\end{equation}
\end{lemma}

The proof is given in Appendix~\ref{app:shape}. Its purpose here is
to control the integrable endpoint singularity of \(w\): the factor
\(\sqrt{T-t}\) leaves an increasing concave function, which is suited
to the signed quadrature kernel below.

\begin{corollary}[Endpoint bound]\label{cor:5.2}
On the inner branch,
\begin{equation}\label{eq:5.3}
\wmin\le w(t)\le\frac P{\sqrt{T-t}}.
\end{equation}
\end{corollary}

Put \(C=P/\wmin\). The increasing function
\(\theta^2/(1-\cos\theta)=2\{(\theta/2)/\sin(\theta/2)\}^2\) gives
\begin{equation}\label{eq:5.4}
\frac\mu\pi<C^2<\frac{\mu\pi}{8}.
\end{equation}
We shall also use
\begin{equation}\label{eq:5.5}
\theta\ge\pi/3\quad\Longrightarrow\quad C<2\sqrt y.
\end{equation}
Indeed, put \(g=\pi y/\mu=\theta-1+\sec\theta-\tan\theta\).
Since \(g'=\sin\theta/(1+\sin\theta)>0\),
\[
g\ge1-\sqrt3+\pi/3>47/150>5/16>\pi^2/32.
\]
Here \(\pi>157/50\), \(\sqrt3<26/15\), and \(\pi^2<10\) suffice. The same angle ratio as in \eqref{eq:5.4} gives
\[
\frac{C^2}{y}=\frac{\theta^2}{2(1-\cos\theta)g}
<\frac{\pi^2}{8g}<4,
\]
proving \eqref{eq:5.5}.

\subsection{Regular cells}\label{sec:5.2}

For integers \(1\le j\le n<T+1/4\), define
\[
E_j=\int_{j-1}^j x(t)\,dt-x(j-1/4),
\qquad Q_n=2\sum_{j=1}^n E_j-n.
\]
A cell is \textbf{complete} if its right endpoint is at most \(T\).
The last cell is \textbf{terminal} if it extends past \(T\); the zero
extension then supplies its unused part.
A cell \textbf{crosses the interface} if its portion below \(T\)
contains \(c\) in its interior.
We call a cell \textbf{regular} if it does not cross the interface.

Because there is only one interface, at most one cell crosses it.
If a terminal cell crosses, its left endpoint \(m=\lfloor T\rfloor\)
satisfies \(m<c<T\); inclusion of its node requires \(T-m>3/4\).
If \(c\le3/4\), every included sampling node lies at or beyond
the interface.

For an affine inverse with slope \(-\wmin\), each cell has
\(E_j=\wmin/4\), so \(2E_j-1=\eta>0\).
Regular cells will retain this margin; the exceptional crossing cell
may use part of the accumulated margin.

We shall repeatedly use a simple signed-integral comparison. Suppose a kernel \(k\) is nonnegative before \(s=a\) and nonpositive afterwards, and \(\omega\ge0\) is decreasing. Then, whenever the integrals exist,
\[
\int k(s)\omega(s)\,ds\ge\omega(a)\int k(s)\,ds,
\]
because \(k(s)(\omega(s)-\omega(a))\ge0\) on both sides. The ratio comparisons below are applications of this observation.

\begin{lemma}[Regular cells]\label{lem:5.3}
Every complete outer cell, and every complete or terminal inner cell containing its quadrature node, satisfies
\begin{equation}\label{eq:5.6}
E_j>\wmin/4.
\end{equation}
If \(c\ge1\), the first cell satisfies \(E_1>\wmin/3\).
The same first-cell bound holds for the ball action \(G_K\) when \(K>\pi\), with \(\wmin=2\).
\end{lemma}

\begin{proof}
Let \(b=j-1\) be the left endpoint of the cell, and write \(E=E_j\).
For an outer cell, convexity gives \(\int_b^{b+1}x>x(b+1/2)\),
and \(w\ge\wmin\) gives
\(x(b+1/2)-x(b+3/4)\ge\wmin/4\).

For any included cell, put \(l=T-b>3/4\), \(L=\min(1,l)\),
and \(k(s)=s-\mathbf1_{\{s\ge3/4\}}\).
Since \(x(b+s)-x(b+3/4)\) equals \(\int_{b+s}^{b+3/4}w\)
for \(s<3/4\) and \(-\int_{b+3/4}^{b+\min(s,l)}w\)
for \(s>3/4\), Fubini's theorem and the zero extension give
\begin{equation}\label{eq:kernel}
E=\int_0^Lk(s)\,w(b+s)\,ds.
\end{equation}
If the first cell is outer, the decreasing factor \(t^{1/3}w(t)\) and the signed-integral comparison give
\[
E_1\ge(3/4)^{1/3}w(3/4)
\left\{\frac32(3/4)^{2/3}-\frac9{10}\right\}>\frac\wmin3.
\]
Indeed, \((3/4)^{1/3}w(3/4)>c^{1/3}\wmin\ge\wmin\), and the bracket exceeds \(1/3\) because \(9/16>(37/45)^3\).
For the ball, use \(c=K/\pi>1\) and \(\wmin=2\) in the same calculation.

For an inner cell the kernel identity becomes
\[
E=\int_0^L\frac{k(s)H(b+s)}{\sqrt{l-s}}\,ds.
\]
Write \(I_i(l)=\int_0^Ls^ik(s)/\sqrt{l-s}\,ds\), \(i=0,1\). We claim
\begin{equation}\label{eq:5.7}
I_0(l)>\frac1{4\sqrt l},\qquad I_1(l)>0.
\end{equation}
For \(l\ge1\), put \(a_i=\int_0^1s^ik(s)\,ds\), so
\(a_0=1/4\), \(a_1=11/96\).
For \(0<s<1\) and \(0<a\le1\), write
\[
W(s)=(1-s)^{-1/2}-1,\qquad
R_a(s)=\frac{W(as)}{W(s)}
=a\sqrt{\frac{1-s}{1-as}}
\frac{1+\sqrt{1-s}}{1+\sqrt{1-as}}.
\]
Both factors are nonincreasing in \(s\).
The signed-integral comparison gives
\[
\sqrt l\,I_i(l)-a_i
=\int_0^1s^ik(s)W(s/l)\,ds
\ge R_{1/l}(3/4)\bigl(I_i(1)-a_i\bigr)>0,
\]
since \(I_0(1)=1/3\) and \(I_1(1)=3/20\).

For \(3/4<l\le1\), the substitution \(u=l-s\) gives
\[
I_0(l)=\frac43l^{3/2}-2\sqrt{l-3/4}
\ge\frac13>\frac1{4\sqrt l},
\]
where the first expression decreases to \(1/3\). Writing \(s=l-(l-s)\) gives
\[
I_1(l)=lI_0(l)-\frac4{15}l^{5/2}
+\frac23(l-3/4)^{3/2}\ge\frac l{15}>0.
\]
This proves \eqref{eq:5.7} in both ranges.

Let \(H(b)+\beta s\) be the secant line through the values of
\(H(b+s)\) at \(s=0\) and \(s=3/4\).
Concavity places \(H(b+s)\) above this line for \(0\le s\le3/4\)
and below it for \(3/4\le s\le L\). The signs of \(k\) therefore give
\[
E\ge H(b)I_0(l)+\beta I_1(l).
\]
This comparison holds for any nonnegative increasing concave factor \(H\). Here \(H(b)>0\) and \(\beta\ge0\), so \eqref{eq:5.7} gives
\(E>H(b)/(4\sqrt l)\ge\wmin/4\).
All endpoint singularities are integrable. \end{proof}

\subsection{The cell crossing the interface}\label{sec:5.3}

For a complete crossing cell write \(k-1<c<k\le T\), where \(k\) is an integer, and put \(h=k-c\),
\[
E=\int_{k-1}^kx(t)\,dt-x(k-1/4).
\]
The following comparison controls the drop of the inner inverse over a short interval. It is used here and in the first-cell estimate.

\begin{lemma}[Homogeneous inverse drop]\label{lem:5.4}
If \(\theta\le\pi/3\), then
\begin{equation}\label{eq:5.8}
Z(v):=\frac{\mu-x(c+v)}{\wmin}<\frac83v
\qquad(0<v\le c/3).
\end{equation}
In particular, \(T-c>c/3\), so the inner branch exists throughout this interval.
\end{lemma}

\begin{proof}
The geometric comparison below applies for all \(0<\mu<K\),
independently of the condition \(T>1\).
For \(\tau=4c/3\), the rescaling
\[
\widetilde A(z)=\tau^{-1}A(\tau z),\qquad
\widetilde x(t)=\tau^{-1}x(\tau t),\qquad
\widetilde Z(v)=\tau^{-1}Z(\tau v)
\]
has interface height \(\widetilde c=3/4\) and leaves
\(\theta\) and \(\wmin\) unchanged. It therefore suffices to
prove the estimate for \(c=3/4\) and \(v\le1/4\).
Put \(\theta_*=\pi/3\), \(M=\mu\theta/\pi\), and
\[
B(z)=A(\mu-\wmin z)-c,\qquad
B'(z)=q_\theta(1-z/M),
\]
where
\[
q_\theta(\xi)=\int_0^1
\left(1+\frac{\xi^{-2}-1}{\sin^2(\theta u)}\right)^{-1/2}\,du.
\]
Indeed, for \(0<\xi<1\),
\[
\arccos(\xi\cos\theta)-\arccos\xi
=\int_0^\theta
\frac{\xi\sin s}{\sqrt{1-\xi^2\cos^2s}}\,ds
=\theta q_\theta(\xi),
\]
with endpoint values obtained by continuity.
The increasing function
\[
\frac{\tan\theta-\theta}{\theta^3}
=\int_0^1u^2\left(\frac{\tan(\theta u)}{\theta u}\right)^2du
\]
and \eqref{eq:5.1} give
\[
M\theta^2\ge
\frac{3\theta_*^3}{4(\tan\theta_*-\theta_*)}>\theta_*^2;
\]
the last step is \(7\pi/12>\sqrt3\). In particular,
\(M\ge(\theta_*/\theta)^2\ge1\).
For the convex function \(F(r)=(1-r)^{-2}-1\),
\[
F(z/M)\le M^{-1}F(z)\le(\theta/\theta_*)^2F(z).
\]
Together with
\(\sin(\theta u)\ge(\theta/\theta_*)\sin(\theta_*u)\), this proves
\(B'(z)\ge q_{\theta_*}(1-z)\) for \(0\le z\le2/3\).

The comparison shell \((\mu_*,K_*)=(3,6)\) has
minimal inverse slope \(3\) and \(M_*=1\).
Define \(B_*(z)=A_*(3-3z)-A_*(3)\).
Since \(B(0)=B_*(0)=0\), the derivative comparison gives
\(B(z)\ge B_*(z)\).
For \(0\le u\le1/9\), squaring gives
\(\sqrt{1-u}\ge1-9u/16\). Consequently,
\[
A_*(1)\ge\frac1\pi\int_3^6\left(1-\frac9{16R^2}\right)dR
=\frac{93}{32\pi},
\]
and
\[
B_*(2/3)-1/4\ge
\frac{3\pi/4+93/32-3\sqrt3}{\pi}>0,
\]
using \(\pi>25/8\) and \(\sqrt3<7/4\).
The concavity of \(B_*\) and \(B_*(0)=0\) give
\(B(z)\ge B_*(z)>3z/8\) for \(0<z\le2/3\).
Since \(M\ge1\), the point \(\mu-2\wmin/3\) is positive, so
\[
T-c>B(2/3)>1/4.
\]
Given \(0<v\le1/4\), take \(z=8v/3\) and invert
\(B(z)>3z/8=v\) to obtain \(Z(v)<8v/3\).
Scaling back proves both \eqref{eq:5.8} and \(T-c>c/3\).
\end{proof}

\begin{lemma}[Complete crossing cells]\label{lem:5.5}
For a complete crossing cell, the following bounds hold:
\[
\begin{aligned}
\mu\ge1&\quad\Longrightarrow\quad E/\wmin\ge9/32-C/12,\\
\theta\le\pi/3,\ c\ge1&\quad\Longrightarrow\quad E>\wmin/8.
\end{aligned}
\]
\end{lemma}

\begin{proof} First we prove the required bound when \(h\ge1/2\):
\begin{equation}\label{eq:5.9}
h\ge1/2\quad\Longrightarrow\quad E>\wmin/8.
\end{equation}

Put \(a=h-1/4\) and
\[
\psi(v)=\begin{cases}1-h+v,&v<a,\\v-h,&v>a.\end{cases}
\]
The outer contribution is at least \(\wmin(1-h)^2/2\).
The inner slope is \(H(c+v)/\sqrt{y-v}\).

Write \(H(c+v)=\wmin\sqrt y+U(v)\), where \(U\ge0\),
\(U(0)=0\), and \(U(v)/v\) is nonincreasing.
The signed-integral comparison gives
\[
J_y:=\int_0^h\frac{\psi(v)v}{\sqrt{y-v}}\,dv
\ge\sqrt{\frac{h-a}{y-a}}J_h,\quad
J_h=\frac43h^{3/2}-\frac4{15}h^{5/2}-h+\frac1{12}>0.
\]
Here \(J_h''=(1-h)/\sqrt h\ge0\),
\(J_{1/2}=3\sqrt2/10-5/12>0\), and
\(J'_{1/2}=5\sqrt2/6-1>0\).
Applying the comparison to \(U(v)/v\) now gives
\[
\int_0^h\frac{\psi(v)U(v)}{\sqrt{y-v}}\,dv
\ge\frac{U(a)}{a}J_y\ge0.
\]

Since \((1-h)^2/2+\int_0^h\psi=1/4\), discard the nonnegative excess before \(a\). On the negative part, \(y\ge h\) gives
\[
\begin{aligned}
\frac E\wmin
&\ge\frac14-\int_a^h(h-v)
\left(\sqrt{\frac y{y-v}}-1\right)dv\\
&\ge\frac14-\int_0^{1/4}u
\left(\sqrt{\frac h u}-1\right)du
=\frac9{32}-\frac{\sqrt h}{12}>\frac{19}{96}>\frac18.
\end{aligned}
\]
This proves \eqref{eq:5.9}, including \(y=h\) by integrability at the endpoint.

The inner excess \(Y(v)=Z(v)-v\) is convex, nonnegative and vanishes at zero. Write also \(X(u)=\int_0^u\bigl(w(c-u')/\wmin-1\bigr)\,du'\ge0\) for the outer excess, so that
\[
\begin{aligned}
x(c-u)&=\mu+\wmin\bigl(u+X(u)\bigr),
&0\le u\le1-h,\\
x(c+v)&=\mu-\wmin\bigl(v+Y(v)\bigr),
&0\le v\le h.
\end{aligned}
\]
For \(h\le1/2\), set \(a_0=(h-1/4)_+\), \(b_0=\min(h,1/4)\). Integration over the cell, together with \((1-h)^2/2-h^2/2+h-1/4=1/4\), gives the same identity for outer and inner nodes:
\begin{equation}\label{eq:5.10}
E/\wmin=1/4+O+Y(a_0)-\int_0^hY(v)\,dv,\qquad O\ge0.
\end{equation}
For an inner node, \(O=\int_0^{1-h}X\ge0\).

For an outer node, put \(u_0=1/4-h\) and
\(\chi(u)=1-h-u-\mathbf1_{\{u<u_0\}}\). Then
\(O=\int_0^{1-h}X-X(u_0)=\int_0^{1-h}X'(u)\chi(u)\,du\).
Here \(X'\ge0\) increases in \(u\) because \(w\) decreases on the
outer branch. The function \(\chi\) is negative on \([0,u_0)\),
nonnegative on \((u_0,1-h]\), and has total mass \(1/4+h^2/2\).
Hence \(\chi(u)(X'(u)-X'(u_0))\ge0\), and
\(O\ge X'(u_0)(1/4+h^2/2)\ge0\).

If \(\theta\le\pi/3,c\ge1\), then \(k\ge2\) and
\(c=k-h\ge3/2\). Thus \(h\le1/2\le c/3\), and the homogeneous
bound \eqref{eq:5.8} gives \(Y(h)<5h/3<2h\). Convexity gives
\[
\int_0^hY\le\frac h2Y(a_0)+\frac{b_0}{2}Y(h).
\]
Hence \(E/\wmin>1/4-hb_0\ge1/8\). Together with \eqref{eq:5.9}, this proves the small-angle bound.

For \(\mu\ge1\), we prove the remaining bound for every complete
crossing cell:
\[
\mu\ge1\quad\Longrightarrow\quad E/\wmin\ge9/32-C/12.
\]
If the node is inner, \eqref{eq:5.3} gives
\[
E\ge9\wmin/32-P\int_0^{1/4}\frac{u}{\sqrt{T-k+u}}\,du
\ge9\wmin/32-P/12.
\]
If the node is outer, \(h\le1/4\), and \eqref{eq:5.10}, \eqref{eq:5.3} imply
\[
E/\wmin\ge1/4+h^2/2
-C\int_0^h\frac{h-s}{\sqrt{y-s}}\,ds
\ge1/4+h^2/2-2Ch^{3/2}/3.
\]
Since \(C>\sqrt{\mu/\pi}>1/2\), the final expression decreases on \([0,1/4]\), where its endpoint value is \(9/32-C/12\). \end{proof}

\subsection{A terminal cell crossing the interface}\label{sec:5.4}

\begin{lemma}[Terminal crossing]\label{lem:5.7}
Suppose \(\mu\ge1\) and \(n\ge2\) is an integer with \(n-1<c<T<n<T+1/4\). The terminal cell satisfies
\[
E_n/\wmin\ge\frac9{32}-\frac C{6\sqrt2},\qquad K<4\pi n.
\]
\end{lemma}

\begin{proof} Put \(\beta=n-T\in(0,1/4)\). We first note that
\[
r(s):=\frac{f_0(s)}{1-\cos s}<\frac{2s}{3}
\qquad(0<s<\pi/2).
\]
Indeed, the equivalent difference \(2s+s\cos s-3\sin s\) and its first derivative vanish at zero, while its second derivative is \(\sin s-s\cos s>0\).

The bound \eqref{eq:5.3} also holds on any outer part of the negative terminal kernel \(n-1/4\le t<T\).
Put \(s=\arccos(x(t)/K)\). Since \(\mu\le x(t)\),
\[
\frac78\le\frac{n-1/4}{n}<\frac tT
\le r(s)<\frac{2s}{3}.
\]
Thus \(w(t)=\pi/s<16\pi/21<5/2\). On the other hand,
\(P^2>\pi\mu/2\ge\pi/2\) and \(T-t<1/4\), so
\(P/\sqrt{T-t}>2P>5/2\); here \(25/8<\pi<22/7\) suffices.

Put \(b=1/4-\beta\). The function
\(\sqrt b(1/2-4b/3)\) on \([0,1/4]\) attains its maximum at \(b=1/8\).
Thus
\[
\Phi(\beta)=\sqrt b(1/2-4b/3)\le1/(6\sqrt2).
\]
The kernel identity gives
\[
E_n\ge9\wmin/32-P\Phi(\beta),
\]
which proves the first bound.

At the interface, \(r(\theta)=c/T>(n-1)/n\ge1/2\), so \(\theta>3/4\).
Since \(\cos(3/4)\le1499/2048<3/4\), we have \(\cos\theta<3/4\).
Together with \(T<n\), this gives
\(K=\pi T/(1-\cos\theta)<4\pi n\).
\end{proof}

\subsection{The first cell}\label{sec:5.5}

\begin{lemma}[First cell]\label{lem:5.8}
For every \(T>1\), we have \(Q_1>0\).
\end{lemma}

\begin{proof} If \(c\ge1\), the first cell is outer and Lemma~\ref{lem:5.3} applies.
Suppose \(c\le3/4\).

Put \(f_T(t)=\sqrt{1-(t/T)^2}\) on \([0,T]\), and \(e_1=\int_0^1f_T(t)\,dt-f_T(3/4)\).
Using the kernel \(k\) and decreasing ratio \(R_a\) from Lemma~\ref{lem:5.3},
\[
\begin{aligned}
T^2e_1-\frac{11}{96}
&=\int_0^1s k(s)W(s^2/T^2)\,ds\\
&\ge R_{1/T^2}(9/16)
\left(\frac{\pi-\sqrt7}{4}-\frac{11}{96}\right)>0.
\end{aligned}
\]
Here \(\pi>25/8\) and \(\sqrt7<8/3\). Thus \(e_1>11/(96T^2)>1/(12T^2)\).

Put \(C_0=\sqrt{\mu K}\) and \(b=\pi T/2\).
Cauchy--Schwarz in the integral defining \(A\) gives
\[
A(z)/T\le\sqrt{1-z^2/(\mu K)}\quad(z\le\mu),
\qquad x(t)\le C_0f_T(t)\quad(t\ge c),
\]
because the mean of \(R^{-2}\) on \([\mu,K]\) is
\(1/(\mu K)\). Also
\[
\int_0^Tx=\int_0^KA=(K^2-\mu^2)/8,
\qquad (K+\mu)/2=\sqrt{C_0^2+b^2}.
\]
Since \(c\le3/4\), compare at the node and on the complementary
interval \([1,T]\). The area identity then gives
\[
Q_1\ge 2C_0e_1+
b\bigl(\sqrt{C_0^2+b^2}-C_0\bigr)-1.
\]
For \(0<\zeta<b\), a one-variable minimization gives
\[
\min_{u\ge0}\left\{\zeta u+
b\bigl(\sqrt{u^2+b^2}-u\bigr)\right\}
=b\sqrt{\zeta(2b-\zeta)}.
\]
The minimum is attained at
\(u=b(b-\zeta)/\sqrt{\zeta(2b-\zeta)}\).
Take \(\zeta=1/(6T^2)\). Since \(2e_1>\zeta\),
\[
Q_1>\frac\pi{2\sqrt6}\sqrt{\pi T-\zeta}-1
>\frac\pi3-1>0,
\]
where \(\pi T>3\) and \(\zeta<1/3\) suffice.

Suppose \(3/4<c<1\), put \(h=1-c<1/4\), and use the kernel \(k\) from Lemma~\ref{lem:5.3}.
The outer-slope comparison gives
\[
\frac1\wmin\int_0^c k(t)w(t)\,dt
\ge F(c):=\frac35c^2-\frac32c+\frac32(3/4)^{2/3}c^{1/3}.
\]
Indeed, \(t^{1/3}w(t)\) decreases, with value at \(3/4\) greater than \(\wmin c^{1/3}\).
The truncated integral \(\int_0^c k(t)t^{-1/3}dt\) is positive: it is at least its value at \(c=1\), computed in Lemma~\ref{lem:5.3}.

The remaining contribution is negative, with magnitude
\[
N:=\frac1\wmin\int_c^1(1-t)w(t)\,dt=\int_0^h Z(v)\,dv.
\]
If \(\theta\le\pi/3\), then \(h<1/4<c/3\), so
\eqref{eq:5.8} gives \(N<4h^2/3\).
If \(\theta\ge\pi/3\), \eqref{eq:5.3} and \eqref{eq:5.5} give
\[
N<2\sqrt y\int_0^h\frac{h-v}{\sqrt{y-v}}\,dv
\le\frac43h^2.
\]
The last inequality follows because
\(\sqrt{y/(y-v)}\) is nonincreasing in \(y\), so the integral is
largest at \(y=h\).
Consequently \(E_1/\wmin>F(c)-4(1-c)^2/3\).
This expression equals \(61/240>1/4\) at \(c=3/4\) and increases,
since
\[
\frac{d}{dc}\left\{F(c)-\frac43(1-c)^2\right\}
=\frac12\left(\frac3{4c}\right)^{2/3}+\frac76-\frac{22}{15}c
>\frac3{40}>0.
\]
Hence \(E_1>\wmin/4>1/2\), proving \(Q_1>0\).
\end{proof}

\section{Proof of the tail estimates}\label{sec:6}

\subsection{Positivity of the quadrature error}\label{sec:6.1}

Table~\ref{tab:cells} collects the cell bounds of Section~\ref{sec:5}
used below. Here \(C=P/\wmin\) and \(E\) denotes the error of the
crossing cell.

\begin{table}[ht]
\centering\small
\begin{tabular}{@{}llll@{}}
\toprule
cell & hypothesis & bound & source\\
\midrule
first, \(c\ge1\) (outer) & none & \(E_1>\wmin/3\) & Lemma~\ref{lem:5.3}\\
first, \(c<1\) & none & \(2E_1-1>0\) & Lemma~\ref{lem:5.8}\\
complete outer & none & \(E_j>\wmin/4\) & Lemma~\ref{lem:5.3}\\
inner, complete or terminal & node below \(T\) & \(E_j>\wmin/4\) & Lemma~\ref{lem:5.3}\\
complete crossing & \(\theta\le\pi/3\), \(c\ge1\) & \(E>\wmin/8\) & Lemma~\ref{lem:5.5}\\
complete crossing & \(\mu\ge1\) & \(E/\wmin\ge9/32-C/12\) & Lemma~\ref{lem:5.5}\\
terminal crossing & \(\mu\ge1\), \(n\ge2\) & \(E_n/\wmin\ge9/32-C/(6\sqrt2)\), \(K<4\pi n\) & Lemma~\ref{lem:5.7}\\
\bottomrule
\end{tabular}
\caption{Cell bounds used in the proof of Proposition~\ref{prop:3.3}.}
\label{tab:cells}
\end{table}

\begin{proof}[Proof of Proposition~\ref{prop:3.3}]
If \(c<1\), Lemma~\ref{lem:5.8} gives \(Q_1>0\), and every later
cell is wholly inner. Hence \(Q_n\ge Q_1+(n-1)\eta>0\).

If no cell crosses the interface, all cells are regular and
\(Q_n>n\eta>0\).

It remains to consider \(c\ge1\) with a crossing cell \(k\le n\). Then \(k\ge2\), so the first cell is outer. Write \(E=E_k\), \(\varepsilon=\pi/2-\theta\), and \(q_*=n\varepsilon/\pi\). Using \(E_1>\wmin/3\) and \(E_j>\wmin/4\) on the other regular cells gives
\begin{equation}\label{eq:6.1}
Q_n/\wmin>q_*+2E/\wmin-1/3.
\end{equation}
If the crossing is complete and \(\theta\le\pi/3\), then \(q_*\ge n/6\ge1/3\) and Lemma~\ref{lem:5.5} gives \(E>\wmin/8\), proving positivity.

In the remaining cases,
\[
E/\wmin\ge\frac9{32}-\frac C{6\sqrt2},\qquad K<4\pi n.
\]
For a terminal crossing these are exactly Lemma~\ref{lem:5.7}.
For a complete crossing, Lemma~\ref{lem:5.5} gives the stronger lower bound \(E/\wmin\ge9/32-C/12\), while \(\theta\ge\pi/3\) and \(c<n\) give
\[
K=\frac{\pi c}{f_0(\theta)}<4\pi n,
\qquad f_0(\theta)\ge f_0(\pi/3)=\frac{\sqrt3}{2}-\frac\pi6>\frac14.
\]
Since \(\varepsilon>\cos\theta\) and \(\theta^2/(1-\cos\theta)<\pi^2/4\),
\[
\frac{C^2}{q_*}
=\frac{K\theta^2\cos\theta}{2n\varepsilon(1-\cos\theta)}
<\frac{K\pi^2}{8n}<\frac{\pi^3}{2}<16.
\]

Substituting these bounds into \eqref{eq:6.1} and completing the
square gives
\[
\begin{aligned}
\frac{Q_n}{\wmin}
&>q_*+\frac{11}{48}-\frac C{3\sqrt2}\\
&>q_*+\frac{11}{48}-\sqrt{\frac{8q_*}{9}}
=\left(\sqrt{q_*}-\frac{\sqrt2}{3}\right)^2+\frac1{144}>0.
\end{aligned}
\]
\end{proof}

\subsection{Comparison with the ball}\label{sec:6.3}

Write \(D_K,q_K\) for the quantities corresponding to the ball
action \(G_K\), and let \(x_K=G_K^{-1}\), extended by zero above \(K/\pi\).
Let \(E_j^K\) and \(Q_n^K\) denote the quantities defined in
Section~\ref{sec:5.2} with \(x\) replaced by \(x_K\) and \(T\) by \(K/\pi\).
Since \(K>\pi\), the first-cell estimate in Lemma~\ref{lem:5.3} gives \(E_1^K>2/3\).
The extended inverse is convex, and \(-x_K'>2\) below \(K/\pi\). At every included node, its supporting tangent therefore gives
\[
E_j^K:=\int_{j-1}^jx_K(t)\,dt-x_K(j-1/4)
\ge-\frac14x_K'(j-1/4)>\frac12.
\]
Thus Lemma~\ref{lem:3.2} yields
\begin{equation}\label{eq:6.5}
q_K(a)=n\ge1\quad\Longrightarrow\quad
D_K(a,\alpha)>Q_n^K\ge Q_1^K=2E_1^K-1>1/3.
\end{equation}
If \(q(a)\ge1\), then \(q_K(a)\ge q(a)\ge1\).
Since \(A=G_K-G_\mu\) and \(q\le q_K\),
\[
D_A(a,\alpha)\ge D_K(a,\alpha)-2\int_a^\infty G_\mu(z)\,dz
\ge D_K(a,\alpha)-\frac{\mu^2}{4}
>\frac13-\frac{\mu^2}{4}\ge\frac1{12}>0.
\]
This proves Proposition~\ref{prop:3.4} and completes the proof of Theorem~\ref{thm:tail}.

\section{Consequences and discussion}\label{sec:applications}

\subsection{Products}

\begin{corollary}[Products]\label{cor:products}
Let \(\Omega_2\subset\mathbb R^k\), \(k\ge1\), be a nonempty bounded open set. For
every integer \(d\ge2\) and every \(\Lambda>0\),
\[
N_D(A^d_{r,R}\times\Omega_2,\Lambda)
<L_{0,d+k}\,|A^d_{r,R}\times\Omega_2|\,\Lambda^{(d+k)/2}.
\]
For every \(\alpha\in\mathbb R\), the same strict inequality, with
\(d=2\), holds for the operator
\(H_\alpha\otimes1+1\otimes(-\Delta^D_{\Omega_2})\) on
\(A^2_{r,R}\times\Omega_2\).
\end{corollary}

\begin{proof}
We use Laptev's argument \cite[Theorem~2.8]{Laptev}. Write
\(\Omega=A^d_{r,R}\) and let \(\tau_1\le\tau_2\le\cdots\) be the
Dirichlet eigenvalues of \(\Omega_2\). Separation of variables gives
\[
N_D(\Omega\times\Omega_2,\Lambda)
=\sum_{i:\,\tau_i<\Lambda}N_D(\Omega,\Lambda-\tau_i).
\]
If \(\tau_1\ge\Lambda\), the count vanishes. Otherwise
Theorem~\ref{thm:shell} bounds each of the finitely many terms strictly
by \(L_{0,d}|\Omega|(\Lambda-\tau_i)^{d/2}\). Since \(d/2\ge1\), the
Berezin inequality for Riesz means \cite{Berezin}, in the form used in
\cite{Laptev}, gives
\[
\sum_i(\Lambda-\tau_i)_+^{d/2}\le
\frac{\Gamma(d/2+1)}{(4\pi)^{k/2}\Gamma((d+k)/2+1)}\,|\Omega_2|\,
\Lambda^{(d+k)/2},
\]
and \(L_{0,d}\Gamma(d/2+1)(4\pi)^{-k/2}\Gamma((d+k)/2+1)^{-1}=L_{0,d+k}\).
The magnetic case is identical, with Theorem~\ref{thm:magnetic} in
place of Theorem~\ref{thm:shell}.
\end{proof}

\subsection{Enhanced bounds}

The following bound combines Theorem~\ref{thm:shell} with known
two-term asymptotics. Its constant depends on the fixed shell;
no uniformity over radius ratios or dimensions is asserted.

\begin{corollary}[Enhanced Dirichlet bound]\label{cor:enhanced}
For every fixed shell \(\Omega=A^d_{r,R}\), there is a constant
\(C_\Omega>0\) such that
\begin{equation}\label{eq:enhanced}
N_D(\Omega,\Lambda)\le L_{0,d}|\Omega|\Lambda^{(d-1)/2}
               (\sqrt\Lambda-C_\Omega)_+\qquad(\Lambda\ge0).
\end{equation}
\end{corollary}

\begin{proof}
Write \(\lambda_j=\lambda_j^D(\Omega)\), put \(a=L_{0,d}|\Omega|\) and
\(b=\tfrac14L_{0,d-1}|\partial\Omega|>0\).
The two-term shell asymptotics of Guo, Jiang, Wang and Yang
\cite[Theorem~1.1(2) in the author version]{GJWY}
(for \(d=2\), see also \cite{GMWW}) give
\[
N_D(\Omega,\Lambda)=a\Lambda^{d/2}-b\Lambda^{(d-1)/2}
             +O(\Lambda^{d/2-2/3}).
\]
Choose \(\Lambda_0\ge\lambda_1\) so that the difference between the
leading Weyl term and \(N_D(\Omega,\Lambda)\) is at least
\((b/2)\Lambda^{(d-1)/2}\) for \(\Lambda\ge\Lambda_0\), and put
\(J=N_D(\Omega,\Lambda_0)\). By Theorem~\ref{thm:shell},
\[
C_\Omega:=\min\left\{\frac b{2a},\
\min_{1\le j\le J}
\frac{a\lambda_j^{d/2}-j}{a\lambda_j^{(d-1)/2}}\right\}>0.
\]
The \(j=1\) term also shows that \(C_\Omega<\sqrt{\lambda_1}\).
For fixed \(k\ge1\), the function
\(\sqrt\Lambda-(k/a)\Lambda^{-(d-1)/2}\) increases.
If \(\lambda_1\le\Lambda\le\Lambda_0\), take \(k=N_D(\Omega,\Lambda)\)
and compare with \(\Lambda=\lambda_k\); the displayed minimum gives
\eqref{eq:enhanced}. High energies follow from the choice of
\(\Lambda_0\), and the count vanishes below \(\lambda_1\).
\end{proof}

Two-term improvements of P\'olya's inequality with explicit constants
are known for the disk \cite[Theorem~1.5]{FLPS2026} and for balls and
cylinders \cite{GMWZ}. In the terminology of
\cite[Remark after Theorem~1.1, equation~(1.4)]{HJLW},
Corollary~\ref{cor:enhanced} establishes the enhanced Dirichlet
P\'olya inequality for each fixed spherical shell.

\subsection{Cubical perforations}

We use the cubical-removal argument of Jiang and Lin
\cite[\S3.2, Lemmas~3.4 and~3.7]{JL}, whose formulation for domains
satisfying the enhanced Dirichlet inequality is recorded in
\cite[paragraph following Theorem~1.2]{HJLW}.
We include the bracketing proof; the new input is the enhanced
inequality for spherical shells.

\begin{corollary}[Finite cubical perforations]\label{cor:holes}
Let \(\Omega=A^d_{r,R}\), with \(C_\Omega\) as in
Corollary~\ref{cor:enhanced}, and let \(U_1,\ldots,U_M\) be open cubes
of side lengths \(\ell_1,\ldots,\ell_M\), with pairwise disjoint
closures contained in \(\Omega\). If
\begin{equation}\label{eq:holes}
\sum_{j=1}^M\ell_j^{d-1}
<\frac{C_\Omega|\Omega|}{d^{3/2}\pi},
\end{equation}
then \(\Omega'=\Omega\setminus\bigcup_j\overline{U_j}\)
satisfies an enhanced Dirichlet bound of the form
\eqref{eq:enhanced}, and hence P\'olya's inequality.
The condition is independent of the locations of the cubes.
\end{corollary}

\begin{proof}
For a cube \(U\) of side length \(\ell\), count positive integer
lattice points in a ball of radius \(\ell\sqrt\Lambda/\pi\).
Rounding each coordinate upward moves a point by at most
\(\sqrt d\). The associated unit cubes therefore cover the smaller
positive-orthant ball. Taking volumes and applying Bernoulli's
inequality gives
\[
\begin{aligned}
N_D(U,\Lambda)&\ge\frac{\omega_d}{2^d}
              \left(\frac{\ell\sqrt\Lambda}{\pi}-\sqrt d\right)_+^d\\
&\ge L_{0,d}\left(\ell^d \Lambda^{d/2}
       -d^{3/2}\pi\ell^{d-1}\Lambda^{(d-1)/2}\right).
\end{aligned}
\]
Dirichlet bracketing gives
\(N_D(\Omega',\Lambda)+\sum_jN_D(U_j,\Lambda)\le N_D(\Omega,\Lambda)\).
For \(\Lambda\ge C_\Omega^2\), subtract the cube lower bounds from
\eqref{eq:enhanced} to obtain
\[
N_D(\Omega',\Lambda)\le L_{0,d}\left(
|\Omega'|\Lambda^{d/2}-B \Lambda^{(d-1)/2}\right),\qquad
B=|\Omega|C_\Omega-d^{3/2}\pi\sum_j\ell_j^{d-1}>0.
\]
For \(\Lambda<C_\Omega^2\), domain monotonicity gives
\(N_D(\Omega',\Lambda)=0\). Thus the enhanced inequality holds with
\(C_{\Omega'}=\min\{C_\Omega,B/|\Omega'|\}>0\).
\end{proof}

\subsection{Quantitative questions and further applications}

The proof of Theorem~\ref{thm:tail} controls every angular tail and
every lattice translation by quadrature of the inverse action.
One can therefore ask which other radial actions have positive
quadrature errors for every number of included levels.
The regular-cell comparisons in Lemma~\ref{lem:5.3} use
monotonicity of the inverse slope and concavity of its normalization
\(H(t)=\sqrt{T-t}\,w(t)\).
The specific shell geometry enters in the estimates for the first
cell and for cells crossing the change of convexity.
Analogous estimates for other action curves could extend the method
beyond spherical shells.

An explicit lower bound for \(C_\Omega\) in terms of \(r\), \(R\)
and \(d\) would quantify Corollary~\ref{cor:enhanced}.
The constant obtained here is positive, but its definition uses a
finite part of the spectrum and an unspecified threshold in the
two-term asymptotics. An effective bound would give an explicit
correction to the individual eigenvalue inequalities and a geometric
threshold for the cubical perforations in Corollary~\ref{cor:holes}.
Its dependence on the radius ratio in the thin-shell and small-hole
limits is also of interest.

\appendix
\section{A moment proof of the inverse-slope shape}\label{app:shape}

\begin{proof}[Proof of Lemma~\ref{lem:5.1}]
We first prove a moment inequality. For a random variable \(Y\ge1\)
with finite fourth moment, put
\[
h=(\mathbb E(Y+1)^{-1})^{-1},\quad m=\mathbb EY,\quad
m_2=\mathbb EY^2,\quad v=\operatorname{Var}Y,\quad
V=\operatorname{Var}(Y^2).
\]
Then \(0\le D:=m_2-(h-1)^2\le V/2\).
The lower bound follows from \(h\le m+1\) and \(m^2\le m_2\).
For the upper bound, \(d=m+1-h\ge0\) satisfies
\[
d=\frac h{m+1}\mathbb E\frac{(Y-m)^2}{Y+1}\le v/2,
\]
so \(m_2-(h-1)^2=v+2md-d^2\le(m+1)v\). Also
\[
\operatorname{Cov}(Y^2,Y)=\mathbb E[Y(Y-m)^2]+mv\ge(m+1)v;
\]
Cauchy--Schwarz gives \(V\ge(m+1)^2v\). Hence
\[
D\le(m+1)v\le\frac{V}{m+1}\le\frac V2,
\]
which proves the claim.

Fix \(0<z<\mu\), put \(S=T-A(z)\), and use
\[
A(z)=\frac1\pi\int_\mu^K\sqrt{1-z^2/R^2}\,dR,\qquad
Y(R)=(1-z^2/R^2)^{-1/2},
\]
with probability measure
\[
d\nu(R)=\frac{Y(R)-Y(R)^{-1}}{\pi z\sigma(z)}\,dR.
\]
All expectations below are taken with respect to \(\nu\).
For fixed \(z<\mu\), \(Y\) is bounded, so the moment inequality applies.
Using \(\partial_zY=zR^{-2}Y^3\) and \(z^2R^{-2}=1-Y^{-2}\), one finds
\[
\pi S=\int_\mu^K(1-Y^{-1})\,dR,\quad
\pi z\sigma=\int_\mu^K(Y-Y^{-1})\,dR,\quad
\pi z^2\sigma'=\int_\mu^K(Y^3-Y)\,dR,
\]
and \(\pi(z^3\sigma''+2z^2\sigma')=\int_\mu^K(3Y^2-1)(Y^3-Y)\,dR\).
Since \(1-Y^{-1}=(Y-Y^{-1})/(Y+1)\) and \(Y^3-Y=Y^2(Y-Y^{-1})\),
division by \(\pi z\sigma\) gives
\[
z\sigma=Sh,\quad z^2\sigma'=Sh\,m_2,\quad
z^3\sigma''=3Sh(\mathbb EY^4-m_2),\qquad h>2,
\]
where \(h>2\) because \(Y>1\) on \((\mu,K]\).
Since \(t=T-S\) and \(dS/dz=\sigma\), we have \(d/dt=-\sigma^{-1}d/dz\)
on the inner branch, and a direct computation with \(H=\sqrt S/\sigma\) gives
\[
-H''(t)=\frac{\mathcal R(z)}{4S^{3/2}\sigma^5},
\qquad z=x(t),
\]
where all derivatives of \(\sigma\) are evaluated at \(z\), and
\[
\begin{aligned}
\mathcal R(z)
&=\sigma^4+4S\sigma^2\sigma'+4S^2\sigma\sigma''
  -12S^2(\sigma')^2\\
&=\frac{S^4h^2}{z^4}\{h^2+4(h-3)m_2+12V\}.
\end{aligned}
\]
If \(2<h<3\), then \(4(h-3)D\ge-2V\), while for \(h\ge3\)
this term is nonnegative. Thus the bracket satisfies
\[
\begin{aligned}
h^2+4(h-3)m_2+12V
&=(h-2)^2(4h-3)+4(h-3)D+12V\\
&\ge(h-2)^2(4h-3)+10V>0.
\end{aligned}
\]

Thus \(H''<0\). Moreover,
\[
\frac{dH}{dS}=\frac{\sigma^2-2S\sigma'}{2\sqrt S\,\sigma^3}<0,
\]
because \(h\le m+1<2m_2\); hence \(H'>0\).
Continuity at \(c\) gives the first identity in \eqref{eq:5.2}.
The expansions
\[
S(z)=\frac{Tz^2}{2\mu K}+O(z^4),\qquad
\sigma(z)=\frac{Tz}{\mu K}+O(z^3)
\]
give the second identity. The finite limits extend monotonicity
and concavity to \([c,T]\).
\end{proof}

\section*{Funding}

This work was supported by the Department of Education of Guangdong
Province (Grant No.~2025ZDJS099), and  Nanfang College, Guangzhou
(Grant No.~2026BQ010).

\section*{Use of AI tools}

During the preparation of this work, the author used ChatGPT (OpenAI)
to assist with mathematical exploration and proof development,
checking intermediate arguments and calculations, literature searches,
and drafting and revising the text and \LaTeX{} source.
The author takes full responsibility for the results, proofs,
references, and final manuscript.

\end{document}